\documentclass[12pt,reqno,draft]{amsproc} 
\usepackage{amstext,amsmath,amssymb,amsthm}
\usepackage{latexsym}
\usepackage{exscale}
\usepackage{color}

\newcommand{\R}{\mathbb{{R}}}

\newtheorem{theorem}{Theorem}
\newtheorem{lemma}{Lemma}

\theoremstyle{definition}

\newtheorem{example}{Example}

\newtheorem{prop}{Proposition}
\newtheorem{cor}[theorem]{Corollary}
\newcommand{\cal}{\mathcal}

\newcommand{\dsize}{\displaystyle}

\theoremstyle{remark}
\newtheorem{remark}[theorem]{Remark}

\newcommand{\M}{\cal M}
\newcommand{\F}{\cal F}
\renewcommand{\L}{\widehat{L}}

\begin{document}
\title{$L^p$ simulation for measures}
\author{L. De Carli and E. Liflyand}

\address{Department of Mathematics and Statistics, Florida International University, Miami, FL, 33199, USA}
\email{decarlil@fiu.edu}

\address{Department of Mathematics, Bar-Ilan University, 52900 Ramat-Gan, Israel}
\email{liflyand@gmail.com}



\begin{abstract} Being motivated by general interest as well as by certain concrete problems of Fourier Analysis,
we construct analogs of the $L^p$ spaces for measures. It turns out that most of standard properties of the usual
$L^p$ spaces for functions are extended to the measure setting. We illustrate the obtained results by examples
and apply them to obtain a version of the uncertainty principle and an integrability result for the Fourier
transform of a function of bounded variation.   \end{abstract}

\maketitle

\section{Introduction}

Looking through any book devoted to  Fourier analysis or just the table of contents, one will see that
the $L^1$ theory of the Fourier transform or the Hilbert transform goes with the corresponding $L^p$ theory.
This is not the case for the theories of the corresponding transforms for measures, see, e.g., \cite{BN}.
A simple curiosity may force one to wonder where the analogs for measures are hidden. We have not succeeded to find
such a machinery in the literature. However, we have a more concrete reason to be interested in the depository of such
treasures. Let us consider the following example, somewhat sketchy. The cosine Fourier transform of a function of bounded
variation on the half-axis,    to wit $f \in BV(\R_+)$, is
\begin{eqnarray}\label{fc} \widehat{f_c}(x)=\int_0^\infty f(t)\cos (2\pi xt)\,dt. \end{eqnarray}
Let $f$ be locally absolutely continuous on $(0, \infty)$; note that here we use not $\mathbb R_+=[0,\infty)$ but $(0, \infty)$
since it is of considerable importance and generality that we can avoid claiming absolute continuity at the origin. Let in addition, 
$\lim\limits_{t\to\infty}f(t)=0$ and ${\mathcal{H}^o}f'\in L^1(\mathbb R_+).$ Here, for any integrable function $g$ on $\mathbb R_+$,

\begin{eqnarray}\label{dht0}\mathcal{H}^og(x)=\frac{2}{\pi}\int_0^\infty\frac{tg(t)}{x^2-t^2}\,dt\end{eqnarray}
is the Hilbert transform applied to the odd extension of $g$; of course, understood in the principle value sense.
When it is integrable, we will denote the corresponding Hardy space of such functions $g$ by $H_0^1(\mathbb R_+)$.
Then the cosine Fourier transform of $f$ in (\ref{fc}) is Lebesgue integrable on $\mathbb R_+$, with
	
\begin{eqnarray}\label{fcehar} \|\widehat{f_c}\|_{L^1(\mathbb R_+)}\lesssim
		\|f'\|_{L^1(\mathbb R_+)}+\|{\mathcal{H}^o}f'\|_{L^1(\mathbb R_+)}=\|f'\|_{H_0^1(\mathbb R_+)}.\end{eqnarray}
For this result as well as many other more advanced ones, see \cite{L0} (cf. \cite{Fr} and \cite{GM1}; see also \cite[Ch.3]{IL} or more 
recent \cite{L2019}). Recall that the derivative of a function of bounded variation exists almost everywhere and is Lebesgue integrable.
Here and in what follows $\varphi\lesssim \psi$ means that $\varphi\le C\psi$ with $C$ being an absolute constant.

A natural question arises {\it whether we can relax the assumption of absolute continuity}. The first step in an eventual proof
is obvious: we integrate by parts in the Stieltjes sense in (\ref{fc}) and arrive at
$$\widehat{f_s}(x)=-\frac1{2\pi x}\int_0^\infty \sin(2\pi xt)\,df(t).$$
However, if we try to follow the lines of the proof of (\ref{fcehar}) and arrive at a version of Hardy's space with integrable Hilbert 
transform of $df$, we will fail. The point is that the Hilbert transform of $df$ does exist almost everywhere (see, e.g., \cite[\S 8.1 ]{BN}) 
but its integrability leads to absolute continuity, the property that we aimed to remove (see, e.g., \cite{CMR} and references therein).

On the other hand, there is a scale of handy subspaces of $H_0^1(\mathbb R_+)$, for which the integrability of the cosine Fourier transform
is valid, with the norm of $f'$ in one of such spaces on the right-hand side of (\ref{fcehar}). More precisely, for $1<p<\infty,$ set

\begin{eqnarray*}\|g\|_{O_p}=\int_0^\infty\left(\frac{1}{x}\int_{x\le t\le 2x}|g(t)|^pdt \right)^{\frac1p}\,dx. \end{eqnarray*}
Further, for $p=\infty$, let

\begin{eqnarray*}\|g\|_{O_\infty}=\int_0^\infty \operatornamewithlimits{ess\, sup}\limits_{x\le t\le 2x}|g(t)|\,dx.\end{eqnarray*}
Known are (see, e.g., the above sources) the following relations:

\begin{eqnarray}\label{embs1} O_\infty\hookrightarrow O_{p_1} \hookrightarrow O_{p_2} \hookrightarrow H_0^1 \hookrightarrow  L^1\quad (p_1>p_2>1).\end{eqnarray}
Under the above assumptions, there holds

\begin{eqnarray}\label{fceop} \|\widehat{f_c}\|_{L^1(\mathbb R_+)}\lesssim \|f'\|_{O_p(\mathbb R_+)},\end{eqnarray}
provided that the right-hand side is finite for some $p>1$. In fact, a different notation is convenient  for the case where  the $O_p$ norm is 
calculated for the derivative: $\|f\|_{V_p}:=\|f'\|_{O_p}.$ Just this notation is appropriate for further generalization. On the one hand, (\ref{fceop}) 
follows from (\ref{fcehar}) and (\ref{embs1}). On the other hand, a  direct proof for (\ref{fceop}) is given in  \cite{GM1}, where the main ingredient 
is the Hausdorff-Young inequality. To provide similar reasoning for measures $\mu_f$ generated by functions of bounded variation $f$ rather than functions 
(however, we shall write $df$ rather than $d\mu_f$), we need a corresponding extension of the Hausdorff-Young inequality.
And here is the point where our special harmonic analysis comes into play. We do not restrict ourselves to finding immediate tools for the above 
problem but try to establish a kind of general and multivariate theory. A variety of relevant issues will be introduced and studied.


\medskip

\subsection{Basic notions}

\quad

We define an analog of $L^p$ spaces for measures by means of an associated norm. For a given $p\in [1, \infty]$,  we use  the notation $\| \cdot \|_p$ 
to denote the  standard  norm in  $L^p(\R^n)=L^p(\R^n, dx)$,  where by $dx$ we mean the Lebesgue measure.

  We  denote by $ {\cal S}(\R^n)$ the  Schwartz space of rapidly decreasing $C^\infty$ functions,   and either by  ${\cal F}(f)  $ or by
  $\widehat f $ the Fourier transform of a function $f\in {\cal S}(\R^n)$, written $$\widehat f(y)=\int_{\R^n} f(x) e^{-2\pi i x\cdot y}\,dx,$$
where $x\cdot y=x_1y_1+...+x_ny_n$. Recall that $\F :{\cal S}(\R^n)\to {\cal S}(\R^n)$ is  one-to-one, and the inverse Fourier transform is   $\check{f}(y)=   
\widehat f(-y)$. In this paper, we will not distinguish between  Fourier transform and inverse Fourier transform, unless it becomes  necessary.

 For $p\in [1, 2]$, the operator $\F : L^p(\R^n)\to L^{p'}(\R^n)$, with $\frac1p+\frac1{p'}=1$, is bounded, with $\|\widehat f\|_{p'}\leq \|f\|_p$ and equality if $p=2$.
For $  L^p(\R^n)$, $p>2$, the Fourier transform can be defined in the distributional sense as
   $$\langle \widehat f, \psi\rangle= \int_{\R^n} f(x) \widehat \psi(x)\, dx, \quad \psi\in {\cal S}(\R^n);$$
clearly, $\widehat f$  is a function if and only if $f=\widehat g$ for some $g\in L^{p'}(\R^n)$. With this observation in mind, we give the following definition.

\medskip

For a given $p\in [1,\, \infty]$, we let 
\begin{equation}\label{hatp} \widehat L^p (\R^n) = \{f\in L^p(\R^n) \, :  \mbox{ $f=\widehat g$ for some } g\in L^{p'}(\R^n) \}.\end{equation}
\noindent In a natural way, we endow $\widehat{L}^p (\R^n)$ with the norm
\begin{equation}\label{normhat}	\|f\|_{\widehat{L}^p } =\|\widehat f \,\|_{p'}. \end{equation}

\medskip

 With this definition, the Fourier transform
 $$\F: L^p(\R^n)\to (\L^{p'}(\R^n), \ \|\cdot \|_{\L^{p'}})$$
 is  a one-to-one  isometry. When $p\in [1, 2]$,  the Hausdorff-Young inequality yields $\|f\|_{\widehat{L}^p }\leq \|f\|_{p}$, with equality if $p=2$.

We denote by $\M $ the space of sigma-finite   Borel measures  on $\R^n$. 
For every $p\in [1, \infty]$, we define the functional $\| \cdot \| _p^* : \M\to[0, \infty]$ as
\begin{equation} \label{p*norm} \| \mu\|_p^*=\sup\limits_{h\in \widehat{L}^{p'} (\R^n) \, :\atop \|  h\|_{\widehat{L}^{p'}}\le1}\left|\int_{\mathbb R^n} h(t)d\mu(t)\right|;
\end{equation}
we let
\begin{equation} \label{p*norm2} \M_p=  \{  \mu \in \M \ :  \| \mu\|_p^* <\infty \}.   \end{equation}
Note that, for every $\mu\in \M_p$ and every $h\in \L^{p'}(\R^n)$, we have that
\begin{equation}\label{basic}
	\left|\int_{\R^n} h(x) d\mu(x) \right| \leq \| h\|_{\L^{p'}}\| \mu\|_p^*.
\end{equation}

We do not assume that  our measures  are positive, or even real-valued. For definition and properties of non-positive measure see e.g. \cite{H}. With this assumption, the spaces $\M_p$ are vector spaces, and we will prove in Section 2 that  the functional $\| \mu\|_p^*$ is a norm on $\M_p$.

\medskip

\subsection{Structure of the paper}

\quad

With $\| \mu\|_{p}^*$ and $\M_p$ denoted by similarity to $L^p$, we then establish basic properties of these measure spaces.

We will prove in Section 2 that the spaces
$\M_p$ have many properties in common with $L^p$ spaces. We establish the properties of   measures in $\M^p$ spaces   and the properties of functions in  spaces $\L^p(\R^n)$.  Discussing then the Fourier transform of a measure, we establish a Hausdorff-Young type inequality. Further, for the convolution 
of a function and a measure, we prove a Young type inequality for our setting. We mention that the results in Section 2 are supplemented with examples.

\medskip

Section 3 is devoted to applications of the introduced machinery.
 One of them is a development of an  uncertainty principle for measures.  The uncertainty principle in Fourier analysis
 quantifies the intuition that a function and its Fourier transform cannot both be concentrated on small sets. Many examples of this principle can
 be found, e.g., in the book by Havin and J\"oricke \cite{HJ}  and in an article by Folland and Sitaram \cite{FS}.
 In Subsection 3.1, using a quantitative version of a result in  \cite{AB}, we prove  that a  finite measure
 and its Fourier transform cannot both  be supported on sets of finite  Lebesgue measure.
 Recall that a  measure $\mu$ is supported in a set $E\subset \R^n$ if $\mu(F)=0$ whenever $F$ is a measurable set that does not intersect $E$.  
 

\medskip
In conclusion, we formulate and prove an analog of (\ref{fceop}) for functions of bounded variation {\it without assuming absolute continuity}.
This is Theorem \ref{main}. In order to formulate and prove it, as an analog of $V_p$ spaces for functions, we introduce the notion $f\in V_p^*$
for measures, with
$$\|f\|_{V_p^*}=\int_0^\infty \, x^{-\frac1p}\,\|\chi_{(x, 2x)}\mu_f\|_{p }^*\,dx $$
where $\chi_E$ denotes the characteristic function of $E$.
The  product of a measure $\mu$ and a  measurable function $f$  is  the measure defined by $(f \mu)(F)=\int_F f\,d\mu$ for every   
measurable set $F$. 
For $1<p\le2$, our new Hausdorff-Young inequality will be helpful, while for $p>2$, we prove an analog of (\ref{embs1}) and use an embedding argument.

\bigskip

\section{$L^p$ properties of measures}

In this section we establish basic properties of  measures in the spaces $\M_p$ defined in the introduction, with $p\in [1,\infty]$, that mimic those of functions in $L^p$ spaces.  
We   also establish properties of the spaces  $\L^p(\R^n)$ defined in \eqref{hatp}.
 
 %
If $E$ is a measurable subset of $\R^n$,  with $|E|\ne 0$,  we let
\begin{equation} \label{Enorm}\| \mu\|_{p, E}^*= \|\chi_E \mu\|_p^*=\sup\limits_{h\in \widehat{L}^{p'} (\R^n) \, :\atop \|  h\|_{\widehat{L}^{p'}}\le1}\left|\int_{E}  h(t)\,d\mu (t)\right|, \end{equation}
and    $\M_{p, E}=  \{  \mu :  \| \mu\|_{p,E}^* <\infty\} $.

%
We can also define
\begin{equation}\label{1loc}
\M_{1, loc}=  \{  \mu :  \| \mu\|_{1,E}^* <\infty \ \mbox {for every measurable bounded set $E$}\}.
\end{equation}

  The standard Lebesgue measure  and the Delta measures are   notable examples  of $\M_p$ measures.   In the rest of this paper we 
will use  ${\mathcal L}$  (or $dx$  in integration) to denote the standard  Lebesgue measure.  
 
  For a given $a\in\R^n$, we let $ \delta_a$ be the  measure defined as   $ \int_{\R^n }f(x)\,d\delta_a =f(a).$  
  
  \begin{example} \label{Leb} We show that the standard Lebesgue measure  is in     $\M_\infty$ and $\|{\cal L}\|_\infty^*= 1.$

  	Indeed, 
  	$$\|{\cal L}\|_{\infty}^*=  \sup\limits_{h\in \widehat{L}^{1} (\R^n) \, :\atop \|  h\|_{\widehat{L}^{1}=\|\widehat h\|_\infty}\le1}\left|\int_{\R^n}  h(t)\,dt\right| \leq
  	\sup\limits_{h\in {L}^{1} (\R^n) \, :\atop \|  h\|_{1}\le1}\left|\int_{\R^n}  h(t)\,dt\right| \leq 1. $$
  	To prove that equality holds, we can consider $  g=e^{-\pi  |x|^2}$.  It is easy to verify that $\widehat g(x)=g(x)$, and so
  	$g\in  \L^1(\R^n)$ and $  \|g\|_{\widehat{L}^{1}}=\|\widehat g\|_\infty=1$. Since    $1=\widehat g(0)=\int_{\R^n}  g(t)\,dt=\|g\|_1$,  we have that
  	$$
  	\|{\cal L}\|_{\infty}^*\ge \int_{\R^n}  g(t)\,dt=1,
  	$$
  	as desired.   
  	
  \end{example}

  \begin{example}\label{ex1}             We show that  $\delta_a\in \M_p$ only for $p=1$ and $\|\delta_a\|_1=1$.  Indeed, assuming $a=0$ for simplicity, we can easily see  that
  	\begin{align*}  \|\delta_0\|_1^* &=\sup\limits_{h\in {\cal S}(\R^n ):\atop{ \|  h \|_{  \L^\infty}\le1}}\left|\int_{\mathbb R^n}h(t)\, d\delta_0\right|=
  		\sup\limits_{h\in {\cal S}(\R^n ):\atop{ \|\, \widehat  h\, \|_{ 1}\le1}} |h(0)|\\
  		&	= \sup\limits_{h\in {\cal S}(\R^n ):\atop{ \| \widehat  h \|_{ 1}\le1}}\left|\int_{\R^n }\widehat h(x)  \,dx\right|
  		\leq \sup\limits_{h\in {\cal S}(\R^n ):\atop{ \| \,\widehat  h\, \|_{ 1}\le1}}\|\,\widehat h\,\|_1= 1.\end{align*}
  	
  	To prove that equality holds, we can consider the function $g= e^{-\pi |x|^2}$ in the previous example and  verify that
  	$ 
  	\|\delta_0\|_1^*\ge \|\hat g\|_1=1.
  	$
  	An easy variation of this argument shows that 	$\delta_0\not \in \M_p$ if $p>1$. \end{example}

\subsection{H\"older type inequalities}

\quad

We prove the following

\begin{theorem}  \label{Holder}
	{\it	 If $\mu\in \M_p $ and $f\in \L^q (\R^n)$,   and $\frac 1r=\frac 1q+\frac 1p$, then}   
\begin{equation}\label{hold1}
	\|f  \mu\|^*_{r}\leq \|f\|_{\L^q }\| \mu\|^*_{p}.
\end{equation}
\end{theorem}

\begin{proof}
Assume $\|f\|_{\L^q} =1$, or else replace $f$ with $\tilde f= \frac{f}{\|f\|_{\widehat{L}^q}}$.
With the notation previously introduced,
$$
\|  f  \mu\|^*_{r}= \sup_{h\in \L^{r'}(\R^n) :\atop \|h\|_{\widehat{L}^{r'}  }\leq 1}  \left|\int_{\R^n} h(y)f(y)\,d\mu(y)\right|.
$$
Let us show that $hf\in \widehat{L}^{p'} $ and $\|hf\|_{\L^{p'}}\leq 1$. Indeed, $\widehat{hf}=\widehat h*\widehat f$ (standard convolution). 
Since $\frac 1{r}+\frac{1}{q'}=  1+\frac 1{p }$, by Young's inequality  for convolution and  the Hausdorff-Young inequality,
$$
\|hf\|_{\L^{p' }}=\|\widehat{hf}\|_{p }= \|\widehat h*\widehat f\|_p \leq \|\widehat h \|_{r} \| \widehat f\|_{q'}= \|  h\|_{\L^{r' } } \|   f\|_{\L^q  }\leq 1.$$
Thus, $\|hf\|_{\widehat{L}^{p' }}\leq 1$, and so
$$
\|f  \mu\|^*_{r}  \leq
\sup_{k\in \L^{p'}(\R^n) :\atop \|k\|_{\widehat L^{p'}}  \leq 1} \left|\int_{\R^n} k(y)\, d\mu(y)\right|= \| \mu\|_p^*=
\| \mu\|_p^*\,\|f\|_{\L^q},  $$
as required.                          \end{proof}

\begin{remark}
	When  $r=1$, for every $\mu\in \M_p $ and $f\in \L^{p'} (\R^n)$ we have that
	$$
	\|f  \mu\|^*_{1}\leq \|f\|_{\L^{p'} }\| \mu\|^*_{p}.
	$$
	This is the case  of \eqref{hold1} that most closely resembles the standard H\"older's inequality. 
	
	\end{remark}

\begin{cor}\label{incl1}
{\it	Let $E $ be a bounded subset of $\R^n$. Then  $\M_{r, E}\subset \M_{p, E}$ whenever} $1\leq p\leq r\leq \infty $.	\end{cor}

\begin{proof}
Assume  $p<r$, since the case $p=r$ is trivial.  Assume also   $E\subset  Q_R=[-R,\, R]^n$ for some $R>0$. By \eqref{Enorm},
$$
\|     \mu\|^*_{r,E}= \sup_{\|h\|_{\L^{r'  }}\leq 1}  \left|\int_{E} h(y)\, d\mu(y)\right| =\sup_{\|\widehat h\|_{ L^{r }}\leq 1}
\left|\int_{\R^n} \chi_Q (y)h(y) \chi_E(y)\,d\mu(y)\right|.
$$
Let $q=\frac{rp}{r-p}$. Since  $r\ne p$, we have $q<\infty$ and  $q'>1$.
The Fourier transform of  the characteristic function of $Q_R$ is $$\widehat  \chi_{Q_R}(x)=\prod\limits_{j=1}^n\frac{\sin(\pi R x_j)}{\pi x_j},$$
and so $\widehat \chi_{Q_R}(x)\in  L^s(\R^n)$ for every $s>1$.
We have  $\|\widehat \chi_{Q_R} \|_s= C_s^nR^{\frac{n}{s'}}$, where
$$C_s= \left\|\frac{\sin(\pi \cdot)}{\pi \cdot}\right\|_s =\begin{cases} \Big(\frac{2s'}\pi\Big)^{\frac{1}s}, & 1<s<2,\\ \Big(\frac2s\Big)^{\frac{1}{2s}}, & 2\leq s<\infty \\1, &s=\infty,  \end{cases}$$
is independent of $R$. In fact, $C_s$ can be taken $\Big(\frac{2s'}\pi\Big)^{\frac{1}s}$ for all $s<\infty$. This is calculated by minimal means: split
the integral

\begin{equation}\label{ballint} \int_{\mathbb R} \Big|\frac{\sin(\pi t)}{\pi t}\Big|^s\,dt=\frac1\pi \int_{\mathbb R} \Big|\frac{\sin(t)}{t}\Big|^s\,dt \end{equation}
into two, over $|t|\le1$ and over $|t|>1$, and replace $\Big|\frac{\sin(t)}{t}\Big|$ in the first by $1$ and in the second by $\frac1{|t|}$.
However, it is known (see \cite[Lemma 3]{Ball} or \cite[Ch.VI, 7.5]{MaPo}) that for $s\ge2$, the sharp bound for (\ref{ballint}) is $\sqrt{\frac2s}$.	

Applying  Proposition \ref{Holder}  with  $f= \chi_{Q_R}$ and  $\chi_E \mu$ in place of $ \mu$,   we obtain

\begin{equation}\label{ineqE} \|  \mu\|^*_{p, E}\leq \|\chi_{Q_R}\|_{\L^q } \|     \mu\|^*_{r,E} = C_q^n R^{\frac{n}{q'}}\|     \mu\|^*_{r,E},\end{equation}
and so  $\|   \mu\|^*_{p, E} <\infty $ whenever $\|   \mu\|^*_{r, E} <\infty$, as required.	\end{proof}

\medskip
\noindent
\begin{cor}\label{inclusion}
 \it{	For every $p\in [1, \infty]$, we have that 	$\M_p\subset \M_{1, loc}$.  }
\end{cor}

\begin{proof}
Follows from Corollary \ref{incl1} and  (\ref{1loc}).
\end{proof}

\noindent
\begin{cor}\label{C1} \it{The functional $\|\ \|_p^*$ is a norm on $\M_p$ for every $p\in [1,\, \infty]$. }
\end{cor}

\begin{proof} It is trivial to verify that  for every $\mu,\ \sigma\in\M_p$ and every $\lambda\in {\mathbb C}$, 
$$\|\mu+\sigma\|_p^*\leq \|\mu\|_p^*+ \|\sigma\|_p^*, \quad \qquad \|\lambda \mu\|_p^*=|\lambda|\, \|\mu\|_p^*.
$$ 
We now prove that $\|\mu \|_p^*=0$ if and only if $\mu\equiv 0$, in the sense that $\mu(E)=0$ for every $\mu-$measurable set $E$.

  { In order to show that $\mu\equiv 0$, it is enough to verify that $\mu(E)=0$ for every bounded set $E$. }
Let $E$ be bounded and  $\mu-$measurable.  Assume that $E\subset Q_R$ for some $R>0$.  Using \eqref{ineqE} and Proposition  \ref{Propset}, we can see at once that
$$
\mu(E)= \int_E d\mu(x)=\int_{ {Q_R}} \chi_E d\mu(x) \leq \|\chi_{Q_R}\|_{\L^\infty} \|\mu\|_{p,E}^* \leq \|\mu\|_{p }^*=0
$$
and so $\mu(E)= 0$ for every $\mu-$measurable bounded  set $E$.         \end{proof}

\subsection{Properties of $\L^p $ spaces}

\quad

In this sub-section we will establish properties of  the spaces $\L^p(\R^n)$  defined in \eqref{hatp}.
We first   shows how measures of the form $d\mu=fdx$  behave with respect to the norms introduced when   $f\in \L^p$, 
\begin{theorem}\label{funct}
	{\it  Let   $ d\mu=fdx $, with $f\in \widehat{ L}^p (\R^n)$  for some $p\in [1, \infty]$; then   $\mu \in\M_p$ and  $$\| \mu \|_p^*=\|f\|_{\widehat{L}^p }.$$}
\end{theorem}

Before discussing Theorem \ref{funct}, we prove the following

\begin{lemma}
	${\cal S}(\R^n)$ is dense in $\L^p(\R^n)$ for
	every $p\in [1, \infty]$.
\end{lemma}
\begin{proof}
	Since
	${\cal S}(\R^n) \subset \L^p(\R^n)\subset  L^p(\R^n)$   and ${\cal S}(\R^n)$ is dense in $L^p(\R^n)$ for
	every $p\in [1, \infty), $  we can see at once that ${\cal S}(\R^n)$ is also dense in $\L^p(\R^n)$.
	To see that ${\cal S}(\R^n)$ is dense also in $\L^\infty(\R^n)$, we observe that every $f\in \L^\infty(\R^n)$ is the
	image of $g\in L^1(\R^n)$ via the Fourier transform. We  can find functions  $\psi_n\in {\cal S}(\R^n) $ such that
	$\dsize\lim_{n\to\infty}\|\psi_n-g\|_1=0$. But  $$ \|\psi_n-g\|_1= \|\,\widehat{\widehat\psi_n}-\widehat f\,\|_1= \|{\widehat\psi_n}-  f\|_{\L^\infty},
	$$
	and so $\dsize\lim_{n\to\infty} \|{\widehat\psi_n}- f\|_{\L^\infty}=0$.  Since $\widehat\psi_n\in {\cal S}(\R^n)$,
	we have proved that ${\cal S}(\R^n)$  is   dense  in $\L^\infty(\R^n)$.                    \end{proof}

\begin{proof}[Proof of Theorem \ref{funct}] Since  ${\cal S}(\R^n)$ is dense in  $L^p(\R^n)$ and in $\L^{p'}(\R^n) $,
	and  the Fourier transform is one-to-one in ${\cal S}(\R^n)$, we can see at once that
	\begin{align*}
		\|f\|_{\widehat{L}^p } &= \|\widehat f\|_{p'}= \sup\limits_{g\in {\cal S}(\R^n):\atop{ \|g\|_{p}\le1}}\left|\int_{\mathbb R^n} g(t)\widehat f(t)\, dt\right|
		= \sup\limits_{g\in {\cal S}(\R^n):\atop{ \|g\|_{p}\le1}}\left|\int_{\mathbb R^n}\widehat  g(t) f(t)\, dt\right|\\
		&=\sup\limits_{h\in {\cal S}(\R^n):\atop{ \|\widehat h \|_{p}\le1}}\left|\int_{\mathbb R^n}h(t) f(t)\, dt\right|=
		\sup\limits_{h\in {\cal S}(\R^n):\atop{ \|h\|_{\L^{p'} }\le1}}\left|\int_{\mathbb R^n}h(t) f(t)\, dt\right|= \| \mu \|_p^*,
	\end{align*}
	which completes the proof.    \end{proof}

\begin{remark}\label{6} If $d\mu=fdx$ is  as in Theorem \ref{funct}  and $p\in [1,2]$, then there holds $$\| \mu_f\|_p^*=\|f\|_{\widehat{L}^p }=\|\widehat f\|_{p'}\leq  \|f\|_p.$$ \end{remark}

\begin{cor}\label{Propset} Let $E\subset \R^n$ be a (Lebesgue) measurable set.  
	
	a) For every $p\in [1,\infty]$, we have 
	\begin{equation}\label{p-main} \|\chi_E\|_{\L^p }\le   | E |^{\frac 1p}. \end{equation}
	
	b)   For every $1\leq p\leq q\leq \infty $ and every  $\mu\in \M_q$, we have  $$\|\mu\|_{p, E}^*\le  \|\mu\|_q^* | E |^{\frac 1r},\qquad \mbox{where $\frac 1r=\frac 1p-\frac 1q$}. $$  
\end{cor}

We have used the standard convention $\frac 1\infty=0$.  Thus,   \eqref{p-main} yields  $\|\chi_E\|_{\L^\infty}\le 1$, for every set $E$.

\begin{proof}
	We first prove a). When $p\in [1,2]$,  Remark \ref{6}  yields
	 $$\|\chi_E\|_{\L^p } \le \|\chi_E\|_{p } = | E |^{\frac 1p}.$$
	
	Assume now  $p\in (2,\infty)$. By the Hausdorff-Young  inequality, we can see at once that
	$$\{f\in \L^{p'} \ : \ \|f\|_{\L^{p'}}=\|\widehat f\|_p\leq 1\} \subset \{ f\in  L^{p'} (\R^n)\ : \ \|f\|_{p'}\leq 1\}.	$$
	In view of this observation and  Theorem \ref{funct}, we can  let $d\sigma=\chi_E dx$ and write the following chain of inequalities:
	\begin{align}
		\nonumber	\|\chi_E\|_{\L^p }=\|\sigma\|_{p, E}^* &=    \sup_{f\in \L^{p'}(\R^n) : \atop{\|  f\|_{\L^{p'}}\leq 1}}\left|\int_{\R^n} \chi_E(x)f(x)dx\right|
		\\\label{inf}
		&\leq \sup_{f\in L^{p'}(\R^n) : \atop{\|  f\|_{p'}\leq 1}}\left|\int_{\R^n} \chi_E(x)f(x)dx\right| \\\nonumber	\label{inf2} &\leq |E|^{\frac 1p} \| f\|_{p'} \leq |E|^{\frac 1p}.
	\end{align}
	We have used   H\"older's inequality  in the last step.  
	
	When $p=\infty$, it follows from \eqref{inf} that
	$$
	\sup_{f\in L^{1}(\R^n) : \atop{\|  f\|_{1}\leq 1}}\left|\int_{\R^n} \chi_E(x)f(x)dx\right|\leq \sup_{f\in L^{1}(\R^n) : \atop{\|  f\|_{1}\leq 1}} \int_{\R^n} \chi_E(x)|f(x)|dx\leq 1.
	$$

	Part b) follows  from H\"older's inequality \eqref{Holder} and  part  a).  Indeed, letting  $r=\frac{pq}{q-p}$, we have
	$$
	\|\mu\|_{p, E}^*=  \|\chi_E\mu\|_{p}^* \leq  \|\chi_E\|_{\L^r} \|\mu\|_{q}^*  \leq |E|^{\frac 1r}  \|\mu\|_{q}^*.
	$$
	The proof of the corollary is complete.      \end{proof}

\medskip

\medskip
We  use Theorem \ref{funct} to prove  inclusion relations of the $\L^p$ spaces and their duals.  Recall that the dual of a normed space $X$, denoted by $(X)'$, 
 is  the set of  linear functionals $L: V  \to {\mathbb C}$   such that $\sup_{ \|f\|_{X}\leq 1} |L(f)| <\infty$.  

By definition,  $\L^p(\R^n)= L^p(\R^n)$ when $p\in [1,2]$ but  in general  $ \L^p(\R^n)$ is a proper subspace of $ L^p(\R^n)$.
For example, the Riemann-Lebesgue Lemma yields that $ \L^\infty(\R^n) $ is a space of  uniformly  continuous functions that go  to zero at infinity.

Even though $L^p(\R^n)=\L^p(\R^n)$ when $p\in [1,2]$, the norms on these spaces are different and so the   duals of these spaces  are different too.
When  $p\leq 2$,    the Hausdorff-Young inequality yields,
$$
\|f\|_{\L^p}= \|\hat f\|_{p'}\leq \|f\|_p.
$$
When $p=2$ we have  $\|  f\|_2=\|f\|_{\L^2}$   but  when $p> 2$  the inequality above can be strict.   

%
%
We prove the following
\begin{prop}\label{dual}{\it 
	For every $p\in [1, \infty]$, we have
	$$
	\L^{p'}(\R^n) \subset ( \L^p(\R^n))'. 
	$$
	When $p\in [1,2]$, we have $
	\L^{p'}(\R^n) \subset ( \L^p(\R^n))'\subset L^{p'}(\R^n). 
	$}
\end{prop}

\begin{proof}
	For a given $g\in \L^{p'}(\R^n)$,  we  let  $d\mu =gdx$  
	and we  let $ L_g:\L^p(\R^n)\to {\mathbb C}$,
	$$L_g(f)=\int_{\R^n} f(x)g(x)dx.
	$$
	By H\"older's inequality   \eqref{hold1}  and Theorem \ref{funct}
	$$|L_g(f)|= \left|\int_{\R^n} f(x)g(x)dx\right|\leq \|f\|_{\L^p}\|\mu \|_{ p'}^*=\|f\|_{\L^p}\|g \|_{\L^{p'}}
	$$ and so $L\in ( \L^p(\R^n))'$.
	
	When $p\leq 2$,  for every  $L\in ( \L^p(\R^n))'$, we have that 
	$$|L(f)|  \leq C\|f\|_{\L^p}=C\|\hat f\|_{p'} \leq C\|f\|_p
	$$
	and so     $L\in ( L^p(\R^n))'= L^{p'}(\R^n)$.  
\end{proof}

\subsection{Fourier transform of finite measures}

The  Fourier transform of a finite  Borel measure $\mu$ is  the function defined as
\begin{equation}\label{def-FT}
	\widehat{ \mu}(y)=\int_{\R^n} e^{-2\pi i x\cdot y}d\mu(x).
\end{equation}
To distinguish it from the Fourier transform for functions, it is sometimes called the Fourier-Stieltjes transform.
It is well-known (see, e.g., \cite[\S 5.3]{BN} or \cite[\S 4.4]{RS}) that the function $\widehat{ \mu}$ is continuous and bounded.
By the Riemann-Lebesgue Lemma, the Fourier transform of an $L^1$ function vanishes at infinity,  but the Fourier transform
of a  $\M_1$ measure  does not need to do so. For example, we have shown in Example \ref{ex1}  that the Delta measure $\mu=\delta_a $
is in $\M_1$; its Fourier transform is   $\widehat{ \mu}(x)=e^{2\pi i a\cdot x}$, and  $ |\widehat{ \mu}(x)|\equiv 1$.

We  prove the following analog of the Hausdorff-Young inequality.
\begin{prop}\label{HY}
	{\it  Let  $ \mu\in  \M_{p}$,   with $1\leq p \leq 2 $.   Then,  $\widehat{ \mu}\in  {L}^{p'} (\R^n)$, and}
	\begin{equation}\label{HY-ineq}
		\|\widehat{ \mu}\|_{p'}\leq \| \mu\|_p^*.
\end{equation}                       \end{prop}

\begin{proof}
	We have observed that the Fourier transform of a finite measure is always  bounded, so the  proposition is trivial for $p=1$. When $p\in (1,2]$, we have 
	$$\|\widehat{ \mu}\| _{p'} =\sup_{h\in {\cal S}(\R^n) :\atop \|h\|_{p}\leq 1}\int_{\R^n} h (y)\widehat{ \mu}(y)\,dy.$$
	By Fubini's theorem,
	\begin{equation}\label{change}
		\int_{\R^n} h (y)\widehat{ \mu}(y)\,dy=\int_{\R^n}\int_{\R^n} h(y) e^{-2\pi i x\cdot y}\,d\mu(x)\,dy=\int_{\R^n} \widehat h(x)\,  d\mu(x).
	\end{equation}
	In view of \eqref{change} and  the fact that  $\|\,\widehat  h\,\|_{p'}\leq \|h\|_p\leq 1$, we can see at once that
	$$
	\|\widehat{ \mu}\| _{p'} =\sup_{h\in{\cal S}(\R^n) : \atop \|h\|_{p}\leq 1}\int_{\R^n} \widehat h (y)\, {\mu}(y)
	\leq \sup_{k\in{\cal S}(\R^n) : \atop \|  \, \widehat k\,\|_{p'}\leq 1} \left|\int_{\R^n} k(x) \, d\mu(x)\right| =\| \mu\|_{p}^* ,$$
	which completes the proof.   \end{proof}

\medskip

\begin{example} If $ \mu_f$ is generated by the singular function $f$ in \cite{WW},   we have
	
	$$|\widehat{\mu_f}(x)|=O\Big(\frac1{|x|^\delta}\Big)$$
	for $|x|$ large, with $0<\delta<\frac12$. Then $\widehat{\mu_f}\in L^{p'}(\R),$ with $\frac1\delta<p'<\infty$, and
	correspondingly, $2<p'<\infty$. By this, $\|\mu_f\|_p^*<\infty$, since
	
	\begin{align*} \|\mu_f\|_p^*&=\sup\limits_{\|h\|_{\L^{p'}}\le1}\left|\int_{\mathbb R} h(t)\,df(t)\right|
		=\sup\limits_{\|g\|_{L^p}\le1}\left|\int_{\mathbb R} \widehat{g}(x)\,{df}(x)\,\right|\\
		&=\sup\limits_{\|g\|_{L^p}\le1}\left|\int_{\mathbb R} {g}(x)\widehat{df}(x)\,dx\right|<\infty,\end{align*}
	because of $g\in L^p$ and $\widehat{\mu_f}\in L^{p'},$ with $1<p<\frac1{1-\delta}<2.$
	
	We have used the pioneer example of a singular function in \cite{WW} but there are more subtle ones. However,
	for all of them there is a barrier to $L^2$, like $0<\delta<\frac12$ above; see, e.g., \cite{Kor} and references therein.
\end{example}

\medskip
 \subsection{Convolution of a function and a measure}

Let $\mu$ be a sigma-finite Borel measure, and let $f:\R^n\to\R$ be a measurable function such that  the function 
\begin{equation}\label{convfm} x\to  \int_{\R^n} f(x-y) d\mu(y)\end{equation} 
is finite for a.e. $x\in \R^n$. The convolution of $f$ and $ \mu$,   denoted by  $f* \mu$, is the function defined in (\ref{convfm}). 
We prove the following analog of the  Young inequality for convolution.

\begin{prop}
{\it	If $\mu\in\M_p$ and $f\in \L^q(\R^n)$ with $\frac 1p+\frac 1q=\frac 1r$,  then $f* \mu\in \L^r(\R^n)$ and}
	$$\|f* \mu\|_{\L^r}\leq \|f\|_{\L^q}\| \mu\|_p^*.$$
	\end{prop}
\begin{proof}
	%
	In view of  Proposition \ref{funct},
\begin{equation}\label{14}\|f* \mu\|_{\L^r} =\|f* \mu\|_{r}^*=\sup_{h\in {\cal S}(\R^n) :\atop \|h\|_{\L^{r' }}\leq 1}\int_{\R^n} h (x)(f* \mu)(x)\,dx.
\end{equation}
For every $h\in{\cal S}(\R^n)$,
\begin{equation}\label{step2}
	\int_{\R^n} h (x)(f* \mu)(x)\,dx=\int_{\R^n} h (x)\int_{\R^n}f(x-y)\,d\mu (y)\, dx,
	= \int_{\R^n} h  *\tilde f(y)\,d\mu(y) \end{equation}
where $\tilde g(t)=g(-t)$.
	By \eqref{step2} and \eqref{basic},
\begin{equation}\label {13}
	\int_{\R^n} h *\tilde f(y)\,d\mu(y)\leq \| h *\tilde f\|_{\L^{p'}} \| \mu\|_p^*=
	\| \widehat h \,\widehat  f\|_{p}\, \| \mu\|_p^*.
\end{equation}
Recalling that  $1+\frac 1r=\frac 1{p }+\frac 1{q }$, we have
	  $\frac 1{p }=\frac {1}{r }+\frac {1}{q'}$. By H\"older's inequality,
	$$
	\|\widehat h \,\widehat  f\|_{p} \leq \|\,\widehat h\,\|_r \|\,\widehat f\,\|_{q'}=\|h\|_{\L^{r'}} \|f\|_{\L^{q}}
	$$
	By \eqref{14} and \eqref{13}, we conclude that
	$$\|f*\mu\|_{\L^r}\leq \|f\|_{\L^{q}}\| \mu\|_p^*,$$ 
as required.	\end{proof}

\section{Applications}

As mentioned, in this section we present applications of the obtained results.

\subsection{Uncertainty principle}

\quad

In this subsection, we  show that the  uncertainty principle   has its embodiment also for measures. We prove the following
 \begin{theorem}\label{Unc}  A   finite  nonzero measure $\mu\in\M_{2} $  and its  Fourier transform $\widehat \mu$ cannot both be supported in   sets of  finite Lebesgue measure.
\end{theorem}


The proof of the theorem relies on the following 
\begin{lemma}\label{uncertainty}
	Let $E , F\subset\R^n$ be    sets   of finite Lebesgue  measure.  There exists a constant $C>0$ such that for every  measure $\mu\in\M_{2}$, we have
	$$\|d\mu\|_{2, F}^*\leq C \|\,\widehat{\mu}\,\|_{L^2( E^c)}.	$$
	\end{lemma}
	
 	\begin{proof} Recall  the following   quantitative form of an uncertainty principle result obtained by Amrein and Berthier in \cite{AB}:
{\it  Let $E, \ F \subset\R^n$ be sets  of finite measure.
	There exists a constant $C>0$ such that for every function} $f\in L^2(\R^n)$,   
\begin{equation}\label{fin}	\|\,\widehat f\,\|_{L^2(F)}\leq  C \|f\|_{L^2(E^c)}.\end{equation}
Let $h\in L^2(\R^n)$. By (\ref{fin}), 
the inequality $\| h\|_{L^2(E^c)} \leq 1 $ yields 
$\|\,\widehat h\,\|_{L^2(F)}\leq  C$.   In view of	 \eqref{change}, we can write the following chain of inequalities:
	 \begin{align*} & \|\,\widehat {\mu}\,\|_{L^2(E^c)}=\sup_{h\in  L^2(\R^n) : \atop \| h\|_{L^2(E^c)}\leq 1} \int_{\R^n }    h (x) \widehat {\mu}(x)\\
	 	=&\sup_{h\in  L^2(\R^n) : \atop \| h\|_{L^2(E^c)}\leq 1} \int_{\R^n }    \widehat h (y) \,  {d\mu}(y)
	 	\ge \sup_{h\in  L^2(\R^n) : \atop \| \,\widehat h\,\|_{L^2(F )}\leq C} \int_{\R^n }    \widehat h (x) \,  {d\mu}(x)\\
	 	=&\frac 1C  \sup_{k\in  L^2(\R^n) :\atop \|\, \widehat k\,\|_{2}\leq 1} \int_{F }    k (x)   {d\mu}(x)=\frac 1C\| {\mu}\|_{2,F  }^*, \end{align*}
	 obtaining the required result.   \end{proof}

\begin{proof}[Proof of Theorem \ref{Unc}]
Assume by contradiction that   $\mu$   is supported in $F$ and  $\widehat \mu$  is supported in $E $, where $E, \ F \subset\R^n$ are
both  of finite measure. By  Lemma \ref{uncertainty}, we have $\|\mu\|^*_{2,F}=\|\chi_F\mu\|^*_{2 }=0$, and  Corollary \ref{C1} yields   $\chi_F\mu\equiv 0$.
Since $\chi_{F^c}\mu\equiv 0$ is assumed, we have $\mu=0$, which is a contradiction. 	\end{proof}

\medskip

\subsection{The Fourier transform theorem}

\quad

In order to reveal an analogy to the case of absolutely continuous $f$, we prove a counterpart of corresponding embeddings in (\ref{embs1}).

\begin{prop}\label{embvps} {\it For $p_1>p_2>1$, there holds}

$$ V^*_{p_1} \hookrightarrow V^*_{p_2}.$$     \end{prop}
\begin{proof} We are going to apply Corollary \ref{incl1}. Since for $E=(x,2x)$, we have in (\ref{ineqE}) that by (\ref{p-main}), there holds
$$ \|\chi_{Q_R}\|_{\L^q } \lesssim x^{\frac1q},$$    and it follows that
$$\|\mu_f \|_{p_2, (x, 2x) }^*\lesssim x^{\frac1q} \|\mu_f \|_{p_1, (x, 2x) }^*.$$

The corresponding relation $\frac1{p_2}=\frac1q+\frac1{p_1}$ yields $\frac1q=\frac{p_1-p_2}{p_1p_2}$.
It remains to observe that
$$ x^{-\frac1{p_2}}\,x^{\frac{p_1-p_2}{p_1p_2}}=x^{-\frac1{p_1}},$$
which leads to the needed embedding.  \end{proof}

With these embeddings and the tools elaborated before, we study, for $\gamma=0$ or $\frac14$, the Fourier transforms
\begin{eqnarray}\label{ftmain} \widehat {f}_\gamma(x)=\int_0^\infty f(t)\cos2\pi(xt-\gamma)\,dt.   \end{eqnarray}
It is clear that $\widehat {f}_\gamma$ represents the cosine Fourier transform in the case $\gamma=0$, while taking $\gamma=\frac14$ gives the sine Fourier transform.

\begin{theorem}\label{main} Let $f$ be of bounded variation on $\R_+$ and vanishing at infinity, that is,
	$\lim\limits_{t\to\infty}f(t)=0$. If $f\in V_p^*$, then for $x>0$, we have
	
	\begin{eqnarray*} \widehat {f}_\gamma(x)= \frac{1}{2\pi x}f\Big(\frac1x\Big)\sin2\pi\gamma+\Gamma(x),\end{eqnarray*}
	where $\gamma=0$ or $\frac14$, and $\|\Gamma\|_{L^1(\bf R_+)}\lesssim\|f\|_{V_p^*}$ provided that the last value is finite for some $p$, $1<p\le\infty$. \end{theorem}

\begin{proof} Splitting the integral in (\ref{ftmain}) and integrating by parts, we obtain

\begin{align*} \widehat {f}_\gamma(x)&=-\frac{1}{2\pi x}f\Big(\frac1x\Big)\sin2\pi(1-\gamma)\\
+\int_0^{\frac1x} f(t)\cos2\pi(xt-\gamma)\,dt &-\frac{1}{2\pi x}\int_{\frac1x}^\infty
\sin2\pi(xt-\gamma)\,df(t).   \end{align*}
Further,
\begin{align*}&\quad\int_0^{\frac1x} f(t)\cos2\pi(xt-\gamma)\,dt\\
&=\int_0^{\frac1x}[ f(t)-f\Big(\frac1x\Big)] \cos2\pi(xt-\gamma)\,dt+\int_0^{\frac1x} f\Big(\frac1x\Big)\cos2\pi(xt-\gamma)\,dt\\
&=-\int_0^{\frac1x}\Big[\int_t^{\frac1x}\,df(s)\Big]\cos2\pi(xt-\gamma)\,dt\\&+\frac{1}{2\pi x}f\Big(\frac1x\Big)
\sin2\pi(1-\gamma)+\frac{1}{2\pi x}f\Big(\frac1x\Big)\sin2\pi\gamma\\
&=\frac{1}{2\pi x}f\Big(\frac1x\Big)\sin2\pi\gamma+\frac{1}{2\pi x}f\Big(\frac1x\Big) \sin2\pi(1-\gamma)+O\biggl(\int_0^{\frac1x}s|df(s)|\,\biggr).
\end{align*}
To continue the proof, we need the following

\begin{lemma}\label{embst} We have the inequality
\begin{eqnarray}\label{embvar} \int_0^\infty|df(s)|\lesssim \|f\|_{V_p^*}.\end{eqnarray}
\end{lemma}

\begin{proof}  There holds
\begin{align*} \ln2\,\int_0^\infty|df(s)|&=\int_0^\infty\frac1x\int_x^{2x}|df(s)|\,dx\\
&=\int_0^\infty\, x^{-\frac1p}\biggl|\int_x^{2x}h(s)\,df(s)\biggr|\,dx,\end{align*}
where $\displaystyle{h(s)=x^{-\frac1{p'}}{\rm sign}\,df(s)}$ if $x<s<2x$ and zero otherwise. This $h$ is not necessarily of
bounded variation; however, since it will always be under the integral sign, we can take an equivalent function
that is of bounded variation. This is possible because the number of jumps of $f$ is of measure zero. We will
continue to use notation $h$ for such a function. It is easy to see that $\|h\|_{p'}=1.$ Let $g(u)=\widehat{h}(u).$ We have
\begin{align*} \int_0^\infty |g(u)|^p\,du&=\biggl(\int_0^{\frac1x}+\int_{\frac1x}^\infty\biggr)|g(u)|^p\,du\\
\lesssim \frac1x \left(\frac{x}{x^{\frac1{p'}}}\right)^p&+x^{-\frac{p}{p'}}\int_{\frac1x}^\infty\,\biggl(\,\biggl|\,h(s)\frac{e^{-ius}}{-iu}\Big|_x^{2x}\biggr|
+\frac1{u}\int_x^{2x} |dh(s)|\,\biggr)^p\,du.\end{align*}
The first term on the right is bounded. Since
$$\int_{\frac1x}^\infty \frac{du}{u^p}\lesssim x^{\frac1{p'}},$$
the definition of $h$ leads to the boundedness of the second term as well.

Therefore, $h$ is the Fourier transform of an $L^p$ function $g$. This leads to the needed right-hand side in (\ref{embvar}).        \end{proof}

\medskip

We return to the proof of the theorem. Since
\begin{eqnarray*}\int_0^\infty\int_0^{\frac1x}s|df(s)|\,dx =\int_0^\infty|df(s)|,       \end{eqnarray*}
it follows from (\ref{embvar}) that to prove the theorem it remains to estimate
\begin{eqnarray*}\int_0^\infty\frac{1}{x}\left|\,\int_{\frac1x}^\infty \sin2\pi(xt-\gamma)\,df(t)\right|\,dx.   \end{eqnarray*}
We have
\begin{align*}  {\ln2}&\int_0^\infty\frac{1}{x}\biggl|\,\int_{\frac1x}^\infty\sin2\pi(xt-\gamma)\,df(t)\biggr|\,dx\\
\le&\int_0^\infty\frac{1}{u}\int_{\frac1u}^\infty\frac{1}{x}\left|\int_u^{2u}\sin2\pi(xt-\gamma)\,df(t)\right|\,dx\,du+
\ln2\int_0^\infty\frac{1}{x}\int_{\frac1x}^{\frac2x}|df(t)|\,dx.
   \end{align*}
The latter summand on the right is controlled by $\int_0^\infty|df(t)|$. Applying H\"older's inequality to the integral in $x$ of the
first summand, we have to estimate
\begin{align}\nonumber&\int_0^\infty\frac{1}{u}\biggl(\int_{\frac1u}^\infty  x^{-p}dx\biggr)^{\frac1p}\,\biggl(\int_0^\infty
\biggl|\int_u^{2u}\sin2\pi(xt-\gamma)\,df(t)\biggr|^{p'}\,dx\biggr)^{\frac1{p'}}\,du
\\\label{B1}
&=\int_0^\infty u^{-\frac 1p} I(u)\, du.
\end{align}
where by $I(u)$ the term in the second parenthesis is denoted.  We can see that
\begin{align*}
I &=\frac 12\left( \int_0^\infty
\left|\int_\R \left(e^{2\pi i(xt-\gamma)}-e^{-2\pi i(xt-\gamma)}\right)\,\chi_{(u, 2u)}(t)\,df(t)\right|^{p'}\,dx\right)^{\frac1{p'}}\\
&=\frac 12\left( \int_0^\infty\left|  \, e^{2\pi i\gamma} \widehat{ \chi_{(u, 2u)}\mu_f} (x)  -e^{-2\pi i\gamma}
\widehat{ \chi_{(u, 2u)}\mu_f}(-x)\,\right|^{p'}\,dx\right)^{\frac1{p'}}\\
&\leq \|\widehat{ \chi_{(u, 2u)}\mu_f}\|_{p'}.
\end{align*}
For $1<p\le2$, applying the   Hausdorff-Young  inequality (\ref {HY-ineq}), we obtain
  $$
 I\leq  \|\widehat{ \chi_{(u, 2u)}\mu_f}\|_{p'} \leq  \|  \chi_{(u, 2u)}\mu_f \|_{p }^*,
$$
from which we derive that \eqref{B1} is bounded by
$$
\int_0^\infty u^{-\frac 1{p}} \|\mu_f \|_{p, (u, 2u) }^*\,du,
$$
as desired. For $p>2$, Proposition \ref{embvps} completes the proof.      \end{proof}

\begin{remark} There exist analogs of (\ref{fceop}) for the multivariate setting; see, e.g., \cite{L0} or \cite{L2019}.
However, the above one-dimensional result is more transparent and illustrative in the sense that extending it to several dimensions
is a plain business with awkward notation and technicalities. \end{remark}

\end{document}